\documentclass[11pt]{article} 
\usepackage{amsfonts,amsmath,latexsym,amssymb,mathrsfs,amsthm,comment}
\usepackage{caption}

\let\OLDthebibliography\thebibliography
\renewcommand\thebibliography[1]{
  \OLDthebibliography{#1}
  \setlength{\parskip}{1pt}
  \setlength{\itemsep}{0pt plus 0.0ex}}

\def\numberlikeadb{\global\def\theequation{\thesection.\arabic{equation}}}
\numberlikeadb
\newtheorem{theorem}{Theorem}[section]
\newtheorem{lemma}[theorem]{Lemma}
\newtheorem{corollary}[theorem]{Corollary}

\newtheorem{proposition}[theorem]{Proposition}
\newtheorem{remark}[theorem]{Remark}

\usepackage{color}

\usepackage{lscape}
\usepackage{caption}
\usepackage{multirow}
\allowdisplaybreaks
\begin{document} 

\title{Absolute moments of the variance-gamma distribution 
}
\author{Robert E. Gaunt\footnote{Department of Mathematics, The University of Manchester, Oxford Road, Manchester M13 9PL, UK}  
} 

\date{} 
\maketitle

\vspace{-10mm}

\begin{abstract} We obtain exact formulas for the absolute raw and central moments of the variance-gamma distribution, as infinite series involving the modified Bessel function of the second kind and the modified Lommel function of the first kind. When the skewness parameter is equal to zero (the symmetric variance-gamma distribution), the infinite series reduces to a single term. Moreover, for the case that the shape parameter is a half-integer (in our parameterisation of the variance-gamma distribution), we obtain a closed-form expression for the absolute moments in terms of confluent hypergeometric functions. As a consequence, we deduce new exact formulas for the absolute raw and central moments of the asymmetric Laplace distribution and the product of two correlated zero mean normal random variables, and more generally the sum of independent copies of such random variables.



\end{abstract}


\noindent{{{\bf{AMS 2010 Subject Classification:}}} Primary 60E05; 62E15; Secondary 33C10, 33C15, 33C20

\section{Introduction}

The variance-gamma (VG) distribution with parameters $\nu > -1/2$, $0\leq|\beta|<\alpha$, $\mu \in \mathbb{R}$, which we denote by $\mathrm{VG}(\nu,\alpha,\beta,\mu)$, has probability density function (PDF)
\begin{equation}\label{vgpdf} p(x) = M \mathrm{e}^{\beta (x-\mu)}|x-\mu|^{\nu}K_{\nu}(\alpha|x-\mu|), \quad x\in \mathbb{R},
\end{equation}
where 
\[M=M_{\nu,\alpha,\beta}=\frac{(\alpha^2-\beta^2)^{\nu+1/2}}{\sqrt{\pi}(2\alpha)^\nu \Gamma(\nu+1/2)},\]
and $K_\nu(x)$ is a modified Bessel function of the second kind. In this parameterisation, $\nu$ is a shape parameter, $\alpha$ is a scale parameter, $\beta$ is a skewness parameter, and $\mu$ is a location parameter. Other parametrisations are given in \cite{gaunt vg,kkp01,mcc98}. 
 Following the seminal works \cite{mcc98,madan}, the VG distribution has been widely used in financial modelling.
Due to the flexibility offered by the four parameters and the modified Bessel function of the second kind in the PDF, the VG distribution is often well-suited to statistical modelling in other application areas. For example, it was reported by \cite{s73} that the VG distribution provided an excellent fit for modelling the size of diamonds mined in South West Africa. The VG distribution also appears as an exact distribution in connection to sample covariances, Wishart matrices and Mat\'ern covariance functions \cite{vgsurvey}, has a well-known connection to the VG process (also referred to as Laplace motion; see, for example, \cite{kkp01}) and has recently found application in time series modelling \cite{ar1}, and in probability theory as a natural limit distribution \cite{aet21,gaunt vg}. These and further application areas and distributional properties are given in
the review \cite{vgsurvey} and Chapter 4 of the book \cite{kkp01}.

In this paper, we will contribute to the distributional theory of the VG distribution by obtaining new exact formulas for the absolute moments of the VG distribution. Absolute moments are a fundamental distributional property, and the topic of establishing exact formulas for absolute moments of probability distributions has received some attention in the literature, with examples including the normal distribution \cite{w12} (expressed in terms of the confluent hypergeometric function of the first kind), the generalized hyperbolic distribution \cite{barn1} (expressed in terms of an infinite series involving the modified Bessel function of the second kind), and the multivariate normal distribution \cite{o21} (very general formulas involving complex summations over multiple indices). Absolute moments arise in numerous applications, including bounding probabilities via Markov's inequality, quantitative probabilistic limit theorems (for example, the Berry-Esseen theorem), and even complicated formulas for absolute moments can yield insights in applications areas; see, for example, \cite{barn1} in which absolute moments of the normal inverse Gaussian (NIG) distribution were used to gain insight into the apparent scaling behaviour of NIG L\'evy processes. Exact formulas for absolute moments in terms of known special functions, which are implemented in modern computational algebra packages, allow for efficient and high precision computation, and allow researchers to take advantage of the extensive theory of the special functions collected in standard references such as \cite{olver} in applications of these formulas.

An exact formula for the moments of the VG distribution in terms of a finite sum involving gamma functions is given by \cite{scott}, which allows lower order moments to be calculated efficiently. Recently, \cite{gaunt23} obtained the following formula for the moments of the $\mathrm{VG}(\nu,\alpha,\beta,0)$ distribution, which simplified an earlier formula of \cite{ha04}.
Let $r\in\mathbb{Z}^+$, and define $\ell:=\lceil r/2\rceil+1/2$ and $m:=r\,\mathrm{mod}\,2$. 
Then, the moments of $X\sim\mathrm{VG}(\nu,\alpha,\beta,0)$ can be expressed in terms of the hypergeometric function: 
\begin{equation}\label{gam}\mathbb{E}[X^r]=\frac{2^{r}(2\beta/\alpha)^m(1-\beta^2/\alpha^2)^{\nu+1/2}}{\sqrt{\pi}\alpha^{r}\Gamma(\nu+1/2)}\Gamma(\nu+\ell)\Gamma(\ell)\,{}_2F_1\bigg(\ell,\nu+\ell;\frac{1}{2}+m;\frac{\beta^2}{\alpha^2}\bigg).
\end{equation}
In addition, \cite{gaunt23} obtained the following formula for the absolute raw moments of the $\mathrm{VG}(\nu,\alpha,\beta,0)$ distribution.
Let $X\sim\mathrm{VG}(\nu,\alpha,\beta,0)$. Then, for $r>\mathrm{max}\{-1,-2\nu-1\}$,
\begin{align}
\mathbb{E}[|X|^r]=\frac{2^r(1-\beta^2/\alpha^2)^{\nu+1/2}}{\sqrt{\pi}\alpha^{r}\Gamma(\nu+1/2)}\Gamma\Big(\nu+\frac{r+1}{2}\Big)\Gamma\Big(\frac{r+1}{2}\Big)\,{}_2F_1\bigg(\frac{r+1}{2},\nu+\frac{r+1}{2};\frac{1}{2};\frac{\beta^2}{\alpha^2}\bigg).\label{hamq}
\end{align}

In this paper, we address a natural problem by deriving exact formulas for the absolute raw moments of the VG distribution for the full range of parameter values without the restriction that the location parameter $\mu=0$. This is significant because in application areas such as financial modelling the location parameter is typically non-zero; see, for example, \cite{s04}. From our formulas for the raw absolute moments of the VG distribution, we immediately deduce  exact formulas for the absolute central moments $\mathbb{E}[|X-\mathbb{E}[X]|^r]$ of the VG distribution (see Remark \ref{rem1}).  Exact formulas for the absolute central moments of the VG distribution have previously not been available for the case of a non-zero skewness parameter $\beta$, even in the special case $\mu=0$. We provide  exact formulas for the absolute moments of the VG distribution for the full range of parameter values in Proposition \ref{prop0} and Theorem \ref{thm1}, and provide simpler closed-form formulas for the \emph{symmetric variance-gamma} distribution ($\beta=0$) in Corollary \ref{cor1} and for half-integer shape parameter $\nu$ in Theorem \ref{thm2}. Our general formula in Theorem \ref{thm1} is expressed in terms of an infinite series involving the modified Bessel function of the second and the modified Lommel function of the first kind; comments on the convergence
of the series are given in Remark \ref{rem9} and a simple asymptotic approximation of the absolute moments is given in Proposition \ref{prop1}. We remark that formulas for the absolute raw moments of the generalized hyperbolic distribution (with includes the VG distribution as a special case) have been obtained by \cite{barn1}; however, these formulas are given for the case that the location parameter is zero (the formulas are equivalently stated as the absolute moments about the location parameter).

The $\mathrm{VG}(1/2,\alpha,\beta,\mu)$ distribution corresponds to the asymmetric Laplace distribution (see the excellent book \cite{kkp01} for a comprehensive account of the distributional theory and application areas), so from Theorem \ref{thm2} we immediately obtain formulas for the absolute raw and central moments of all order of the asymmetric Laplace distribution (see Section \ref{sec3}). Surprisingly, our general formula appears to be new. Moreover, when specialised to the case of the absolute first central moment, our formula corrects an erroneous formula of \cite{kkp01}.
More generally, the sum of independent and identically distributed Laplace random variables is VG distributed, so we also obtain formulas for the absolute moments of such random variables.  
The product of two correlated zero mean normal random variables, and more generally the sum of independent copies of such random variables, are also VG distributed \cite{gaunt prod}, and so we also immediately deduce new exact formulas for the absolute central moments of these distributions. 
These distributions themselves have numerous applications that date back to the 1930's with the work of \cite{craig,wb32}; see \cite{gaunt22,np16} for an overview of application areas and distributional properties.

\section{Absolute moments of the variance-gamma distribution}

We begin by noting that the absolute moments of the VG distribution take a simple form in the case $r\in2\mathbb{Z}^+=\{2,4,6,\ldots\}$. 

\begin{proposition}\label{prop0}Let $X\sim \mathrm{VG}(\nu,\alpha,\beta,\mu)$, where $\nu > -1/2$, $0\leq|\beta|<\alpha$, $\mu\in\mathbb{R}$. Let $p:=\lceil k/2\rceil+1/2$ and $q=k\,\mathrm{mod}\,2$. Then, for $r\in2\mathbb{Z}^+$,
\begin{align*}
\mathbb{E}[|X|^r]=\frac{(1-\beta^2/\alpha^2)^{\nu+1/2}}{\sqrt{\pi}\alpha^r\Gamma(\nu+1/2)}\sum_{k=0}^r\binom{r}{k}2^k(\alpha\mu)^{r-k}\bigg(\frac{2\beta}{\alpha}\bigg)^q\Gamma(\nu+p)\Gamma(p)\,{}_2F_1\bigg(p,\nu+p;\frac{1}{2}+q;\frac{\beta^2}{\alpha^2}\bigg).  
\end{align*}
\end{proposition}

\begin{proof}
Let $X\sim \mathrm{VG}(\nu,\alpha,\beta,\mu)$. Then it is easily seen that $X=_d Y+\mu$, where $Y\sim \mathrm{VG}(\nu,\alpha,\beta,0)$. Therefore, since $r$ is even, we have that $\mathbb{E}[|X|^r]=\mathbb{E}[(Y+\mu)^r]=\sum_{k=0}^r\binom{r}{k}\mu^{r-k}\mathbb{E}[Y^k]$, and using the formula (\ref{gam}) for $\mathbb{E}[Y^k]$ now yields the desired formula.  
\end{proof}

\begin{remark}\label{rem1}
The mean of $X\sim \mathrm{VG}(\nu,\alpha,\beta,\mu)$ is given by 
$\mathbb{E}[X]=\mu+(2\nu+1)\beta/(\alpha^2-\beta^2)$.
Combining this formula with the fact that $X+a\sim\mathrm{VG}(\nu,\alpha,\beta,\mu+a)$, we have that
\[X-\mathbb{E}[X]\sim\mathrm{VG}\bigg(\nu,\alpha,\beta,-\frac{(2\nu+1)\beta}{\alpha^2-\beta^2}\bigg).\]
Therefore a formula for the absolute central moments of even order $r$ follows from setting $\mu=-(2\nu+1)\beta/(\alpha^2-\beta^2)$ in the formula of Proposition \ref{prop0}. In the same manner, we can immediately deduce formulas for the absolute central moments of the VG distribution from all other formulas for the absolute raw moments of the VG distribution that are given in this paper.
\end{remark}

In virtue of Proposition \ref{prop0}, we will henceforth mostly focus on the case that $r\geq1$ is an odd number. The following theorem provides an exact formula for the absolute raw moments of odd order of the VG distribution for the full range of parameter values. For $\mu\geq\nu>-1/2$, let
\begin{align}\label{gmn2}G_{\mu,\nu}(x)=x\big(K_{\nu}( x)\tilde{t}_{\mu-1,\nu-1}( x)+K_{\nu-1}( x)\tilde{t}_{\mu,\nu}( x)\big),
\end{align}
where $\tilde{t}_{\mu,\nu}(x)$ is a normalisation of the modified Lommel function of the first kind that was introduced by \cite{gaunt lommel},
\begin{align*}\tilde{t}_{\mu,\nu}(x)
=\frac{1}{2^{\mu+1}\Gamma((\mu-\nu+3)/2)\Gamma((\mu+\nu+3)/2)} {}_1F_2\bigg(1;\frac{\mu-\nu+3}{2},\frac{\mu+\nu+3}{2};\frac{x^2}{4}\bigg).
\end{align*}
In interpreting the formula in the theorem, it should be noted that, for fixed $\mu\geq\nu>-1/2$, $G_{\mu,\nu}(x)$ is an increasing function of $x$ on $(0,\infty)$ satisfying $0<G_{\mu,\nu}(x)<1$, $x>0$, with $G_{\mu,\nu}(0)=0$ (see \cite{gaunt24}). We also have the special cases
\begin{align}\label{twice}
G_{\nu,\nu}(x)&=x\big(K_{\nu}( x)\mathbf{L}_{\nu-1}( x)+K_{\nu-1}( x)\mathbf{L}_{\nu}( x)\big),   \quad
G_{\nu+1,\nu}(x)=1-\frac{x^{\nu+1}K_{\nu+1}(x)}{2^\nu\Gamma(\nu+1)},
\end{align}
where $\mathbf{L}_\nu(x)$ is the modified Struve function of the first kind. The first formula in (\ref{twice}) follows because $\tilde{t}_{\nu,\nu}(x)=\mathbf{L}_\nu(x)$ and the second formula is obtained as follows. It was shown in \cite{gaunt24} that, for $\mu\geq\nu>-1/2$ and $x>0$,
\begin{align}
\label{370}\int_0^x t^\mu K_\nu(t)\,\mathrm{d}t=2^{\mu-1}\Gamma\bigg(\frac{\mu-\nu+1}{2}\bigg)\Gamma\bigg(\frac{\mu+\nu+1}{2}\bigg)G_{\mu,\nu}(x).
\end{align}
Moreover, in the case $\mu=\nu+1$, formula (\ref{370}) simplifies to $\int_0^x t^{\nu+1} K_\nu(t)\,\mathrm{d}t=2^\nu\Gamma(\nu+1)-x^{\nu+1}K_{\nu+1}(x)$ (which is obtained using \cite[equation 10.43.1]{olver} and the limiting form (\ref{Ktend0})), and the second formula of (\ref{twice}) follows from combining this integral formula with (\ref{370}). Finally, we let $\mathrm{sgn}(x)$ denote the sign function, $\mathrm{sgn}(x)=1$ for $x>0$, $\mathrm{sgn}(0)=0$, $\mathrm{sgn}(x)=-1$ for $x<0$
(the value of $\mathrm{sgn}(x)$ at $x=0$ will, however, be irrelevant in this paper).

\begin{theorem}\label{thm1} Let $X\sim \mathrm{VG}(\nu,\alpha,\beta,\mu)$, where $\nu > -1/2$, $0\leq|\beta|<\alpha$, $\mu\in\mathbb{R}$.  Then, for odd $r\geq1$,
\begin{align}
\mathbb{E}[|X|^r]&=\frac{2^r(1-\beta^2/\alpha^2)^{\nu+1/2}}{\sqrt{\pi}\alpha^{r}\Gamma(\nu+1/2)}\bigg\{\sum_{k=0}^{(r-1)/2}\bigg(\frac{\alpha\mu}{2}\bigg)^{2k}\Gamma\Big(\frac{r-2k+1}{2}\Big)\Gamma\Big(\nu+\frac{r-2k+1}{2}\Big)\nonumber\\
&\quad\times\bigg[\binom{r}{2k}{}_2F_1\bigg(\frac{r-2k+1}{2},\nu+\frac{r-2k+1}{2};\frac{1}{2};\frac{\beta^2}{\alpha^2}\bigg)\nonumber\\
&\quad+\beta\mu\binom{r}{2k+1}{}_2F_1\bigg(\frac{r-2k+1}{2},\nu+\frac{r-2k+1}{2};\frac{3}{2};\frac{\beta^2}{\alpha^2}\bigg)\bigg]\nonumber \\
  &\quad -\sum_{k=0}^{r} \binom{r}{k}(-1)^k\bigg(\frac{\alpha|\mu|}{2}\bigg)^k\sum_{j=0}^\infty\frac{(-2\mathrm{sgn}(\mu)\beta/\alpha)^j}{j!}\Gamma\Big(\frac{r+j-k+1}{2}\Big)\nonumber\\
&\quad\times \Gamma\Big(\nu+\frac{r+j-k+1}{2}\Big)G_{\nu+r+j-k,\nu}(\alpha|\mu|)\bigg\}. \label{formula1} 
\end{align}
\end{theorem}

\begin{remark} It is a natural question to ask for exact formulas for the absolute moments of the $\mathrm{VG}(\nu,\alpha,\beta,\mu)$ distribution for non-integer $r$. Such formulas could be obtained by adapting the proof of Theorem \ref{thm1}. In particular, one would apply the generalised binomial theorem to expand quantities of the form $(t\pm\alpha\mu)^r$ in the integrands of integrals (with respect to the variable $t$). As such, the formulas for absolute moments for non-integer $r$ would be expressed in terms of a double infinite series, rather than a single infinite series as given in formula (\ref{formula1}) of Theorem \ref{thm1} for integer $r$. Also, to ensure that the series expansion
of  $(t\pm\alpha\mu)^r$ is absolutely convergent, one would need to integrate over different intervals than in
the proof of Theorem \ref{thm1}, which would lead to a less clean final formula.
However, as we will see below in Theorem \ref{thm2}, closed-form formulas for non-integer $r$ can be given in the case that $\nu$ is a half-integer.  
\end{remark}


\begin{corollary}\label{cor1}Let $X\sim \mathrm{VG}(\nu,\alpha,0,\mu)$, where $\nu > -1/2$, $\alpha>0$, $\mu\in\mathbb{R}$. Then, for odd $r\geq1$,
\begin{align}
\mathbb{E}[|X|^r]&=\frac{2^r}{\sqrt{\pi}\alpha^{r}\Gamma(\nu+1/2)}\bigg\{\sum_{k=0}^{(r-1)/2}\binom{r}{2k}\bigg(\frac{\alpha\mu}{2}\bigg)^{2k}\Gamma\Big(\frac{r-2k+1}{2}\Big)\Gamma\Big(\nu+\frac{r-2k+1}{2}\Big)\nonumber\\
&\quad -\sum_{k=0}^{r} \binom{r}{k}(-1)^k\bigg(\frac{\alpha|\mu|}{2}\bigg)^k\Gamma\Big(\frac{r-k+1}{2}\Big)\Gamma\Big(\nu+\frac{r-k+1}{2}\Big)G_{\nu+r-k,\nu}(\alpha|\mu|)\bigg\}.\label{rrr}
\end{align}
When $r=1$, we have a further simplification:
\begin{align*}
\mathbb{E}[|X|]&=\frac{1}{\alpha}\bigg[(\alpha\mu)^2\big(K_{\nu}( \alpha|\mu|)\mathbf{L}_{\nu-1}( \alpha|\mu|)+K_{\nu-1}( \alpha|\mu|)\mathbf{L}_{\nu}( \alpha|\mu|)\big)+\frac{(\alpha|\mu|)^{\nu+1}K_{\nu+1}(\alpha|\mu|)}{\sqrt{\pi}2^{\nu-1}\Gamma(\nu+1/2)}\bigg].
\end{align*}
\end{corollary}

\begin{proof}
Formula (\ref{rrr}) follows from setting $\beta=0$ in (\ref{formula1}) and then using that ${}_2F_1(a,b;c;0)=1$. The simplification in the case $r=1$ follows from applying the formulas for $G_{\nu,\nu}(x)$ and $G_{\nu+1,\nu}(x)$ given in (\ref{twice}) to equation (\ref{rrr}).   
\end{proof}

\vspace{2mm}

\noindent \emph{Proof of Theorem \ref{thm1}.} By a change of variables followed by an application of the power series expansion of the exponential function, and interchanging the order of summation and integration, 
\begin{align}
\mathbb{E}[|X|^r]
&=\frac{M}{\alpha^{\nu+r+1}}\int_{-\infty}^\infty |t+u|^r\mathrm{e}^{\gamma t}|t|^\nu K_\nu(|t|)\,\mathrm{d}t=\frac{M}{\alpha^{\nu+r+1}}\sum_{j=0}^\infty\frac{\gamma^j}{j!} I_{j}, \label{sub}
\end{align}
where $I_j=\int_{-\infty}^\infty |t+u|^r t^j|t|^\nu K_\nu(|t|)\,\mathrm{d}t$, $j\geq0$,
and $u=\alpha\mu$ and $\gamma=\beta/\alpha$. 

We will now suppose that $u>0$; the case $u<0$ is similar and is omitted. For $u>0$, after some basic manipulations, we can write
\begin{align*}
I_j
=\int_0^\infty [(t+u)^r+(-1)^j(t-u)^r]t^{\nu+j}K_\nu(t)\,\mathrm{d}t +2(-1)^j\int_0^u(u-t)^r t^{\nu+j}K_\nu(t)\,\mathrm{d}t,
\end{align*}
where we made use of the fact that $r\geq1$ is odd. Using the binomial series
and then evaluating the integrals using (\ref{370}) equation 6.561(16) of \cite{gradshetyn} we obtain that
\begin{align*}
\mathbb{E}[|X|^r]&=\frac{2M}{\alpha^{\nu+r+1}}\bigg[\sum_{j=0}^\infty\frac{\gamma^{2j}}{(2j)!}\sum_{k=0}^{(r-1)/2}\binom{r}{2k}u^{2k}\int_0^\infty t^{\nu+r+2j-2k}K_\nu(t)\,\mathrm{d}t\\
&\quad+\sum_{j=0}^\infty\frac{\gamma^{2j+1}}{(2j+1)!}\sum_{k=0}^{(r-1)/2}\binom{r}{2k+1}u^{2k+1}\int_0^\infty t^{\nu+r+2j-2k}K_\nu(t)\,\mathrm{d}t\\
&\quad-\sum_{j=0}^\infty\frac{(-\gamma)^j}{j!}\sum_{k=0}^r\binom{r}{k}(-u)^k\int_0^u t^{\nu+r+j-k}K_\nu(t)\,\mathrm{d}t\bigg]  \\
&=\frac{2^{\nu+r}M}{\alpha^{\nu+r+1}}\bigg[\sum_{k=0}^{(r-1)/2}\binom{r}{2k}\frac{u^{2k}}{2^{2k}}S_{1,k}+\sum_{k=0}^{(r-1)/2}\binom{r}{2k+1}\frac{u^{2k+1}}{2^{2k+1}}S_{2,k}\\
&\quad-\!\sum_{j=0}^\infty\!\frac{(-\gamma)^j}{j!}\!\sum_{k=0}^r\binom{r}{k}2^{j-k}(-u)^k\Gamma\Big(\frac{r+j-k+1}{2}\Big)\Gamma\Big(\nu+\frac{r+j-k+1}{2}\Big) G_{\nu+r+j-k,\nu}(u)\bigg],
\end{align*}
where
\begin{align*}
S_{1,k}&=\sum_{j=0}^\infty\frac{(2\gamma)^{2j}}{(2j)!}\Gamma\Big(\frac{r+2j-2k+1}{2}\Big)\Gamma\Big(\nu+\frac{r+2j-2k+1}{2}\Big)\\
&=\Gamma\Big(\frac{r-2k+1}{2}\Big)\Gamma\Big(\nu+\frac{r-2k+1}{2}\Big){}_2F_1\bigg(\frac{r-2k+1}{2},\nu+\frac{r-2k+1}{2};\frac{1}{2};\gamma^2\bigg),\\
S_{2,k}&=\sum_{j=0}^\infty\frac{(2\gamma)^{2j+1}}{(2j+1)!}\Gamma\Big(\frac{r+2j-2k+1}{2}\Big)\Gamma\Big(\nu+\frac{r+2j-2k+1}{2}\Big)\\
&=2\gamma\Gamma\Big(\frac{r-2k+1}{2}\Big)\Gamma\Big(\nu+\frac{r-2k+1}{2}\Big){}_2F_1\bigg(\frac{r-2k+1}{2},\nu+\frac{r-2k+1}{2};\frac{3}{2};\gamma^2\bigg),
\end{align*}
which gives formula (\ref{formula1}).
\qed

\vspace{3mm}

For half-integer $\nu$, we obtain a simpler closed-form formula
expressed in terms of the confluent hypergeometric functions $M(a,b,x)$ and $U(a,b,x)$ (see \cite[Chapter 13]{olver}) and the beta function.

\begin{theorem}\label{thm2}
Let $X\sim \mathrm{VG}(\nu,\alpha,\beta,\mu)$, where $\nu=m+1/2$ for $m=0,1,2,\ldots$, and $0\leq|\beta|<\alpha$, $\mu\in\mathbb{R}$. Then, for $r>-1$,
\begin{align}
\mathbb{E}[|X|^r]&=\frac{(1-\beta^2/\alpha^2)^{m+1}}{2^{m+1}m!\alpha^r}\sum_{j=0}^{m}\frac{(m+j)!}{(m-j)!j!2^j}\bigg[\frac{\mathrm{e}^{-\alpha|\mu|-\beta\mu}\Gamma(r+1)}{(1+\mathrm{sgn}(\mu)\beta/\alpha)^{r+m-j+1}}\nonumber\\
&\quad\times U(j-m,j-m-r,\alpha|\mu|+\beta\mu)\nonumber\\
&\quad+(\alpha|\mu|)^{r+m-j+1}B(r+1,m-j+1)M(m-j+1,r+m-j+2,-\alpha|\mu|-\beta\mu)\nonumber\\
&\quad+(\alpha|\mu|)^{r+m-j+1}(m-j)!U(m-j+1,r+m-j+2,\alpha|\mu|-\beta\mu)\bigg], \quad \mu\not=0.  \label{halfv}  
\end{align}
whilst, for $\mu=0$,
\begin{align}
\mathbb{E}[|X|^r]=\frac{(1-\beta^2/\alpha^2)^{m+1}}{2^{m+1}m!\alpha^r}\sum_{j=0}^{m}\frac{(m+j)!}{(m-j)!j!2^j}\bigg[\frac{\Gamma(r+1)}{(1+\beta/\alpha)^{r+m-j+1}}+\frac{\Gamma(r+1)}{(1-\beta/\alpha)^{r+m-j+1}}\bigg].\label{halfmu0}
\end{align}
\end{theorem}

\begin{proof} Suppose first that $\mu\not=0$. Arguing similarly to the proof of Theorem \ref{thm1} and using the elementary representation of the modified Bessel function $K_\nu(x)$ (see \cite[equations 10.49.9 and 10.49.12]{olver}) (since $\nu=m+1/2$ for $m=0,1,2,\ldots$), we obtain that, for $\gamma=\beta/\alpha$ and $u=\alpha\mu>0$,
\begin{align}
 \mathbb{E}[|X|^r]
 &=\sqrt{\frac{\pi}{2}}\frac{M}{\alpha^{r+m+3/2}}\sum_{j=0}^{m}\frac{(m+j)!}{(m-j)!j!2^j}\bigg[\int_0^\infty(t+u)^rt^{m-j}\mathrm{e}^{-(1-\gamma) t} \,\mathrm{d}t\nonumber\\
 &\quad+\int_0^u(u-t)^rt^{m-j}\mathrm{e}^{-(1+\gamma) t} \,\mathrm{d}t+\int_u^\infty(t-u)^rt^{m-j}\mathrm{e}^{-(1+\gamma) t} \,\mathrm{d}t\bigg].\label{asw}
\end{align}
We evaluate the first integral using the integral representation of $U(a,b,x)$ (see \cite[13.4.4]{olver}); the second integral using equation 3.383(1) of \cite{gradshetyn}; and the third integral using equation 3.383(4) of \cite{gradshetyn} expressed in terms of the confluent hypergeometric function $U(a,b,x)$. This yields (\ref{halfv}) for $\mu>0$, and the case $\mu<0$ is similar. To obtain (\ref{halfmu0}) we set $\mu=0$ in (\ref{asw}) and evaluate the two remaining integrals using the integral representation of $U(a,b,x)$ (see \cite[13.4.4]{olver}).
\end{proof}


\begin{remark}\label{rem9} 1. The infinite series in formula (\ref{formula1}) is absolutely convergent. Indeed,
\begin{align*}
&\sum_{j=0}^\infty\frac{(-2\mathrm{sgn}(\mu)\beta/\alpha)^j}{j!}\Gamma\Big(\frac{r+j-k+1}{2}\Big)\Gamma\Big(\nu+\frac{r+j-k+1}{2}\Big)G_{\nu+r+j-k,\nu}(\alpha|\mu|) \nonumber \\
&\quad< \sum_{j=0}^\infty\frac{(2|\beta|/\alpha)^j}{j!}\Gamma\Big(\frac{r+j-k+1}{2}\Big)\Gamma\Big(\nu+\frac{r+j-k+1}{2}\Big)\\
&\quad=\Gamma\Big(\frac{r-k+1}{2}\Big)\Gamma\Big(\nu+\frac{r-k+1}{2}\Big)\,{}_2F_1\bigg(\frac{r-k+1}{2},\nu+\frac{r-k+1}{2};\frac{1}{2};\frac{\beta^2}{\alpha^2}\bigg)\nonumber\\
&\quad\quad+\frac{2|\beta|}{\alpha}\Gamma\Big(\frac{r-k}{2}+1\Big)\Gamma\Big(\nu+\frac{r-k}{2}+1\Big)\,{}_2F_1\bigg(\frac{r-k}{2}+1,\nu+\frac{r-k}{2}+1;\frac{3}{2};\frac{\beta^2}{\alpha^2}\bigg),
\end{align*}
where we used that $0<G_{\mu,\nu}(x)<1$ to get the inequality.
This bound provides insight into the convergence rate, implying geometric convergence, which is faster the smaller the ratio $\beta/\alpha$ is. 

\vspace{2mm}

\noindent 2. The infinite series in formula (\ref{formula1}) also converges faster if the product $\alpha\mu$ is small. This can be seen since, as $\alpha\mu\rightarrow0$,
\begin{align}\label{limg}
\bigg(\frac{\alpha|\mu|}{2}\bigg)^k G_{\nu+r+j-k,\nu}(\alpha|\mu|)\sim \frac{2^{1-\nu-r-j}J_{\nu,r,j,k}(\alpha|\mu|)}{\Gamma((r+j-k+1)/2)\Gamma(\nu+(r+j-k+1)/2)},    
\end{align}
where
\begin{equation*}
J_{\nu,r,j,k}(\alpha|\mu|)=\begin{cases}\displaystyle \frac{2^{\nu-1}\Gamma(\nu)}{r-k+j+1}(\alpha|\mu|)^{r+j+1}, &   \: \nu> 0, \\[10pt]
\displaystyle-\frac{1}{r-k+j+1}(\alpha|\mu|)^{r+j+1}\log(\alpha|\mu|), & \:  \nu = 0, \\[10pt]
\displaystyle\frac{2^{-\nu-1}\Gamma(-\nu)}{2\nu+r-k+j+1}(\alpha|\mu|)^{2\nu+r+j+1}, & \: \nu\in(-\tfrac{1}{2},0). \end{cases} 
\end{equation*}
The limiting form (\ref{limg}) is obtained from the integral representation (\ref{370}) of $G_{\nu+r+j-k,\nu}(x)$ and then performing a simple asymptotic analysis of the integral using the limiting form 
\begin{equation}\label{Ktend0}K_{\nu} (x) = \begin{cases} 2^{\nu -1} \Gamma (\nu) x^{-\nu}+o(x^{-\nu}), & \quad x \downarrow 0, \: \nu > 0, \\
-\log x+O(1), & \quad x \downarrow 0, \: \nu = 0, \\
2^{-\nu -1} \Gamma (-\nu) x^{\nu}+O(x^{-\nu}), & \quad x \downarrow 0, \: \nu\in(-\tfrac{1}{2},0).
\end{cases} 
\end{equation}
(see \cite[Chapter 10]{olver}). One can get further insight from the following asymptotic approximation.
\end{remark}

\begin{proposition}\label{prop1} Let $X\sim \mathrm{VG}(\nu,\alpha,\beta,\mu)$, where $\nu > -1/2$, $0\leq|\beta|<\alpha$, $\mu\in\mathbb{R}$. Let $r\geq1$ be an odd integer. Then, as $\alpha\mu\rightarrow0$,
\begin{align*}
\mathbb{E}[|X|^r]&=  \frac{2^r(1-\beta^2/\alpha^2)^{\nu+1/2}}{\sqrt{\pi}\alpha^{r}\Gamma(\nu+1/2)}\bigg\{\sum_{k=0}^{(r-1)/2}\bigg(\frac{\alpha\mu}{2}\bigg)^{2k}\Gamma\Big(\frac{r-2k+1}{2}\Big)\Gamma\Big(\nu+\frac{r-2k+1}{2}\Big)\\
&\quad\times\bigg[\binom{r}{2k}{}_2F_1\bigg(\frac{r-2k+1}{2},\nu+\frac{r-2k+1}{2};\frac{1}{2};\frac{\beta^2}{\alpha^2}\bigg)\\
&\quad+\frac{\beta}{\alpha}(\alpha\mu)\binom{r}{2k+1}{}_2F_1\bigg(\frac{r-2k+1}{2},\nu+\frac{r-2k+1}{2};\frac{3}{2};\frac{\beta^2}{\alpha^2}\bigg)\bigg] +R_{\nu,\alpha,\beta,r}(\alpha|\mu|)\bigg\},
\end{align*}    
where, for $u=\alpha|\mu|$,
\begin{align*}
R_{\nu,\alpha,\beta,r}(u)=\begin{cases}\displaystyle (r+1)^{-1}\Gamma(\nu) u^{r+1} +o(u^{r+1}), &   \: \nu> 0, \\[1pt]
\displaystyle-2(r+1)^{-1}u^{r+1}\log(u) +O(u^{r+1}), & \:  \nu = 0, \\[1pt]
\displaystyle 2^{-2\nu}\Gamma(-\nu)B(r+1,2\nu+1) u^{r+1+2\nu}+O(u^{r+1}), & \: \nu\in(-\tfrac{1}{2},0). \end{cases}     
\end{align*}
\end{proposition}

\begin{remark} The parameter regimes $|\beta|/\alpha\ll 1$ and $\alpha|\mu|\ll 1$ are often encountered when using the VG distribution to model log returns of financial assets. For instance, in Example 1 of \cite{s04} concerning readings from Standard and Poor's 500 Index we have $\nu=1.870$, $\alpha=271.1$, $\beta=-2.342$ and $\mu=2.585\times10^{-4}$ so that $\beta/\alpha=-8.64\times 10^{-3}$ and $\alpha\mu=7.01\times 10^{-2}$. Recalling Remark \ref{rem1}, for computation of central absolute moments we would take $\mu=-(2\nu+1)\beta/(\alpha^2-\beta^2)$, in which case $\alpha\mu=4.09\times 10^{-2}$.

In such examples, truncating the infinite series in formula (\ref{formula1}) and retaining just the first few terms can result in excellent approximations for absolute raw and absolute central moments of the VG distribution. In fact, even an application of the simple asymptotic approximation given in Proposition \ref{prop1} can yield excellent approximations in such parameter regimes. For example, using \emph{Mathematica}, we applied the asymptotic approximation from Proposition \ref{prop1} to estimate the absolute raw first moments of a VG distribution with the parameter values given in Example 1 of \cite{s04}, and obtained relative error of just $6.60\times10^{-4}$, whilst for the approximation of the absolute central first moment the relative error was $1.09\times10^{-3}$.  As can be seen from the term $R_{\nu,\alpha,\beta,\mu}(\alpha|\mu|)$, the order of the error in the asymptotic approximation decreases as the order $r$ of the moments increases, and for the absolute raw and central third moments the relative errors had an order of magnitude no larger than $10^{-6}$, which was the same order as the rounding error in the numerical integration we used for calculation of the absolute moments. 
\end{remark}



\begin{lemma}\label{lem8}
Let $\nu>-1/2$, $r\geq1$, $j\geq0$ and $u>0$. Let $I_{\nu,r,j}(u)=\int_0^u(u-t)^rt^{\nu+j}K_\nu(t)\,\mathrm{d}t$. Then, as $u\rightarrow0$,   
\begin{align*}
I_{\nu,r,j}(u)=\begin{cases} 2^{\nu-1}\Gamma(\nu)B(r+1,j+1) u^{r+j+1} +o(u^{r+j+1}), &   \: \nu> 0, \\
-B(r+1,j+1)u^{r+j+1}\log(u) +O(u^{r+j+1}), & \:  \nu = 0, \\
 2^{-\nu-1}\Gamma(-\nu)B(r+1,2\nu+j+1) u^{r+j+1+2\nu}+O(u^{r+j+1}), & \: \nu\in(-\tfrac{1}{2},0). \end{cases}    
\end{align*}
\end{lemma}

\begin{proof}
Apply the limiting form (\ref{Ktend0}) to the integrand of the integral $I_{\nu,r,j}(u)$ and then, following a change of variables, evaluate (the leading order term of) the resulting integral using the standard formula $B(a,b)=\int_0^1 x^{a-1}(1-x)^{b-1}\,\mathrm{d}x$.   
\end{proof}

\noindent \emph{Proof of Proposition \ref{prop1}.} 
As usual, we suppose that $u=\alpha\mu>0$; the case $u<0$ is similar and is omitted. On examining the proof of Theorem \ref{thm1}, we see that
\begin{align*}
\mathbb{E}[|X|^r] &=  \frac{2^r(1-\beta^2/\alpha^2)^{\nu+1/2}}{\sqrt{\pi}\alpha^{r}\Gamma(\nu+1/2)}\sum_{k=0}^{(r-1)/2}\bigg(\frac{u}{2}\bigg)^{2k}\Gamma\Big(\frac{r-2k+1}{2}\Big)\Gamma\Big(\nu+\frac{r-2k+1}{2}\Big)\\
&\quad\times\bigg[\binom{r}{2k}{}_2F_1\bigg(\frac{r-2k+1}{2},\nu+\frac{r-2k+1}{2};\frac{1}{2};\gamma^2\bigg)\\
&\quad+\gamma u\binom{r}{2k+1}{}_2F_1\bigg(\frac{r-2k+1}{2},\nu+\frac{r-2k+1}{2};\frac{3}{2};\gamma^2\bigg)\bigg] +\tilde{R}_{\nu,\gamma,r}(u),
\end{align*}
where $\gamma=\beta/\alpha$ and
$
\tilde{R}_{\nu,\gamma,r}(u)=2M\alpha^{-\nu-r-1}\sum_{j=0}^\infty((-\gamma)^j/j!)\int_0^u(u-t)^rt^{\nu+j}K_\nu(t)\,\mathrm{d}t.    
$
Applying Lemma \ref{lem8} to $\tilde{R}_{\nu,\gamma,r}(u)$ as $u\rightarrow0$ completes the proof.
 \qed

\section{Special cases}\label{sec3}


1. As noted by \cite{kkp01}, the $\mathrm{VG}(1/2,\alpha,\beta,\mu)$ distribution corresponds to the asymmetric Laplace distribution, denoted by $\mathrm{AL}(\alpha,\beta,\mu)$, with PDF $p(x)=(2\alpha)^{-1}(\alpha^2-\beta^2)\mathrm{e}^{\beta (x-\mu)-\alpha|x-\mu|}$, $x\in\mathbb{R}$.
Here we specialise formula (\ref{halfv}) to the case $m=0$ to deduce formulas for the absolute moments of the asymmetric Laplace distribution.

Setting $m=0$ in (\ref{halfv}) and using that $U(0,b,x)=1$ (see \cite[equation 13.6.3]{olver}) gives the following formula for the absolute raw moments of the asymmetric Laplace distribution. Let $X\sim\mathrm{AL}(\alpha,\beta,\mu)$. Then, for $r>-1$,
\begin{align}
\mathbb{E}[|X|^r]&=\frac{\alpha^2-\beta^2}{2\alpha^{r+2}}\bigg[\frac{\mathrm{e}^{-\alpha|\mu|-\beta\mu}\Gamma(r+1)}{(1+\mathrm{sgn}(\mu)\beta/\alpha)^{r+1}}+\frac{(\alpha|\mu|)^{r+1}}{r+1}M(1,r+2,-\alpha|\mu|-\beta\mu)\nonumber\\
&\quad+(\alpha|\mu|)^{r+1}U(1,r+2,\alpha|\mu|-\beta\mu)\bigg], \quad \mu\not=0,  \label{halfv2}  \\
\mathbb{E}[|X|^r]&=\frac{\alpha^2-\beta^2}{2\alpha}\Gamma(r+1)\bigg[\frac{1}{(\alpha+\beta)^{r+1}}+\frac{1}{(\alpha-\beta)^{r+1}}\bigg], \quad \mu=0. \label{agg}
\end{align}
Formula (\ref{halfv2}) appears to be new. Formula (\ref{agg}) is in agreement with formula (3.1.26) of \cite{kkp01} for the absolute moments of general order for the asymmetric Laplace distribution with zero location parameter (they equivalently stated their formula as one for the absolute moments about the location parameter). Using that $M(1,b,x)=(b-1)\mathrm{e}^x\gamma(b-1,x)$ (for $b\not=1$) and $U(1,b,x)=\mathrm{e}^{x}x^{1-b}\Gamma(b-1,x)$ (see \cite[equations 13.6.5 and 13.6.6]{olver}), we can express the confluent hypergeometric functions in (\ref{halfv2}) in terms of the lower and upper incomplete gamma functions $\gamma(a,x)=\int_0^x t^{a-1}\mathrm{e}^{-t}\,\mathrm{d}t$ ($a>0$) and $\Gamma(a,x)=\int_x^\infty t^{a-1}\mathrm{e}^{-t}\,\mathrm{d}t$ ($a\in\mathbb{R}$). For $\mu\not=0$ and $r>-1$,
\begin{align}
\mathbb{E}[|X|^r]&=\frac{\alpha^2-\beta^2}{2\alpha}\bigg[\frac{\mathrm{e}^{\alpha|\mu|-\beta\mu}}{(\alpha-\mathrm{sgn}(\mu)\beta)^{r+1}}\Gamma(r+1,\alpha|\mu|-\beta\mu)\nonumber\\
&\quad+\frac{\mathrm{e}^{-\alpha|\mu|-\beta\mu}}{(\alpha+\mathrm{sgn}(\mu)\beta)^{r+1}}\Big(\Gamma(r+1)+(-1)^{-r-1}\gamma(r+1,-\alpha|\mu|-\beta\mu)\Big)\bigg]. \label{zero}
\end{align}
Further setting $r=1$ in (\ref{agg}) and (\ref{zero}) gives the following formula for the absolute first moment of the asymmetric Laplace distribution:
\begin{equation}\label{eq1}
\mathbb{E}[|X|]=\begin{cases}\displaystyle |\mu|+\frac{2\mathrm{sgn}(\mu)\beta}{\alpha^2-\beta^2}+\frac{\alpha-\mathrm{sgn}(\mu)\beta}{\alpha(\alpha+\mathrm{sgn}(\mu)\beta)}\mathrm{e}^{-\alpha|\mu|-\beta\mu}, &    \mu\not=0, \\[10pt]
\displaystyle\frac{1}{2\alpha}\bigg[\frac{\alpha-\beta}{(\alpha+\beta)^2}+\frac{\alpha+\beta}{(\alpha-\beta)^2}\bigg] , & \mu=0. \end{cases}  
\end{equation}
Following Remark \ref{rem1}, on setting $\mu=-2\beta/(\alpha^2-\beta^2)$ in the above formulas (\ref{halfv2}), (\ref{zero}) and (\ref{eq1}) we deduce formulas for the absolute central moments of the $\mathrm{AL}(\alpha,\beta,\mu)$ distribution. In particular, setting  $\mu=-2\beta/(\alpha^2-\beta^2)$ in (\ref{eq1}) 
yields the following formula for the mean deviation of the $\mathrm{AL}(\alpha,\beta,\mu)$ distribution:
\begin{equation}\label{mad}
\mathbb{E}[|X-\mathbb{E}[X]|]=\frac{\alpha+|\beta|}{\alpha(\alpha-|\beta|)}\exp\bigg\{-\frac{2|\beta|}{\alpha+|\beta|}\bigg\}.  
\end{equation}
This formula corrects the erroneous formula (3.1.27) of \cite{kkp01}. Here we will state the corrected formula in their parameterisation, for which $\alpha=(\sqrt{2}\sigma)^{-1}(\kappa^{-1}+\kappa)$, $\beta=(\sqrt{2}\sigma)^{-1}(\kappa^{-1}-\kappa)$,   
where $\sigma>0$ and $\kappa>0$.
The corrected formula is:
\begin{equation*}
\mathbb{E}[|X-\mathbb{E}[X]|]=\begin{cases}\displaystyle \frac{\sqrt{2}\sigma}{\kappa(1+\kappa^2)}\mathrm{e}^{\kappa^2-1}, &    \kappa\in(0,1], \\[10pt]
\displaystyle \frac{\sqrt{2}\sigma}{\kappa^{-1}(1+\kappa^{-2})}\mathrm{e}^{\kappa^{-2}-1}, & \kappa>1. \end{cases} 
\end{equation*}
We also correct the subsequent erroneous formula of \cite[Section 3.1.5]{kkp01} for the ratio of the mean deviation and the standard deviation:
\begin{align*}
\frac{\text{Mean deviation}}{\text{Standard deviation}}=\begin{cases}\displaystyle \frac{2}{(1+\kappa^2)\sqrt{1+\kappa^4}}\mathrm{e}^{\kappa^2-1}, &    \kappa\in(0,1], \\[10pt]
\displaystyle \frac{2}{(1+\kappa^{-2})\sqrt{1+\kappa^{-4}}}\mathrm{e}^{\kappa^{-2}-1}, & \kappa>1, \end{cases}  
\end{align*}
where we used that $\mathrm{Var}(X)=\sigma^2(1+\kappa^4)/(2\kappa^2)$ in the parameterisation of \cite{kkp01}.

We also remark that setting $m=n-1$ and replacing $\mu$ by $n\mu$ in formula (\ref{halfv}) and setting $m=n-1$ in formula (\ref{halfmu0}) gives formulas for the absolute raw moments of the sum $\sum_{i=1}^nX_i$, where $X_1,\ldots,X_n$ are independent $\mathrm{AL}(\alpha,\beta,\mu)$ random variables for the cases $\mu\not=0$ and $\mu=0$, respectively. This follows from the fact that if $Y_1,\ldots,Y_n$ are independent $\mathrm{VG}(\nu,\alpha,\beta,\mu)$ random variables then $\sum_{i=1}^nY_i\sim\mathrm{VG}(n\nu+(n-1)/2,\alpha,\beta,n\mu)$ (which follows from a reparameterisation and simple induction on relation (15) of \cite{bibby}). Therefore if  $X_1,\ldots,X_n$ are independent $\mathrm{AL}(\alpha,\beta,\mu)$ random variables, then $\sum_{i=1}^nX_i\sim\mathrm{VG}(n-1/2,\alpha,\beta,n\mu)$.


\vspace{2mm}

\noindent 2. Let $(U,V)$ be a bivariate normal random vector with  zero mean vector, variances $(\sigma_U^2,\sigma_V^2)$ and correlation coefficient $\rho$. Let $Z=UV$ denote the product of these correlated normal random variables, and write $s=\sigma_U\sigma_V$.  We also consider the mean $\overline{Z}_n=n^{-1}\sum_{i=1}^nZ_i$, where $Z_1,\ldots,Z_n$ are independent copies of $Z$. It was shown by \cite{gaunt vg} that $Z$ is VG distributed and later by \cite{gaunt prod} that 
\begin{equation}\label{vgrep}\overline{Z}_n\sim\mathrm{VG}\bigg(\frac{n-1}{2},\frac{n}{s(1-\rho^2)},\frac{n\rho}{s(1-\rho^2)},0\bigg).
\end{equation}

 Formulas for the absolute central moments of $\overline{Z}_n$ can be immediately obtained by combining (\ref{vgrep}) with (\ref{formula1}) and (\ref{halfv}).
 Formulas for the absolute central moments of $Z$ follow on setting $n=1$. 
We remark that very general and rather complicated formulas for various moments (which includes absolute central moments) of the multivariate normal distribution have recently been obtained by \cite{o21}. However, applying formula (\ref{formula1}) yields a much simpler formula, which is
more convenient to use for computation of absolute central comments of the product $Z$.

\appendix

\section*{Acknowledgements}
The author is funded in part by EPSRC grant EP/Y008650/1. I would like to thank the referee for their constructive comments.


\footnotesize

\begin{thebibliography}{99}
\addcontentsline{toc}{section}{References}

\bibitem{aet21} Azmoodeh, E., Eichelsbacher, P. and Th\"{a}le, C. Optimal Variance-Gamma approximation on the second Wiener chaos. \emph{J. Funct. Anal.} $\mathbf{282}$ (2022), Art.\ 109450.

\bibitem{barn1} Barndorff-Nielsen, O. E. and Stelzer, R. Absolute Moments of Generalized Hyperbolic Distributions and Approximate Scaling of Normal Inverse Gaussian L\'{e}vy Processes. \emph{Scand. J. Stat.} $\mathbf{32}$ (2005),  617--637.

\bibitem{bibby} Bibby, B. M. and S{\o}rensen, M.  Hyperbolic Processes in Finance.  In: Rachev, S. (Ed.), \emph{Handbook of Heavy Tailed Distributions in Finance.} Elsevier Science, Amsterdam (2003),  211--248. 

\bibitem{craig} Craig, C. C. On the Frequency Function of $xy$. \emph{Ann. Math. Stat.} $\mathbf{7}$ (1936), 1--15.

\bibitem{vgsurvey} Fischer, A., Gaunt, R. E. and Sarantsev, A. The Variance-Gamma Distribution: A Review. To appear in \emph{Stat. Sci.}, 2024+.


\bibitem{gaunt vg} Gaunt, R. E.  Variance-Gamma approximation via Stein's method.  \emph{Electron. J. Probab.} $\mathbf{19}$ no. 38 (2014).

\bibitem{gaunt prod} Gaunt, R. E. A note on the distribution of the product of zero mean correlated normal random variables. \emph{Stat. Neerl.} $\mathbf{73}$ (2019),  176--179.

\bibitem{gaunt lommel} Gaunt, R. E.  Bounds for modified Lommel functions of the first kind and their ratios. \emph{J. Math. Anal. Appl.} $\mathbf{486}$ (2020), Art.\ 123893.

\bibitem{gaunt22} Gaunt, R. E. The basic distributional theory for the product of zero mean correlated normal random variables. \emph{Stat. Neerl.} $\mathbf{76}$ (2022), 450--470.

\bibitem{gaunt23} Gaunt, R. E. On the moments of the variance-gamma distribution. \emph{Stat. Probabil. Lett.} $\mathbf{201}$ (2023), Art.\ 109884.

\bibitem{gaunt24} Gaunt, R. E. On the cumulative distribution function of the variance-gamma distribution. To appear in \emph{B. Aust. Math. Soc.}, 2024+.

\bibitem{gradshetyn} Gradshetyn, I. S. and Ryzhik, I. M.  \emph{Table of Integrals, Series and Products,}  $7$th ed.  Academic Press, 2007.

\bibitem{ha04} Holm, H. and Alouini, M.--S. Sum and Difference of two squared correlated Nakagami variates with the McKay distribution. \emph{IEEE T. Commun.} $\mathbf{52}$ (2004),  1367--1376.

\bibitem{ar1} Johannesson, P., Podg\'orski, K., Rychlik, I. and Shariati, N. AR(1) time series with autoregressive gamma variance for road topography modeling. \emph{Probabilist. Eng. Mech.} $\mathbf{43}$ (2016), 106--116.

\bibitem{kkp01} Kotz, S., Kozubowski, T. J. and Podg\'{o}rski, K. \emph{The Laplace Distribution and Generalizations: A Revisit with New Applications.} Springer, 2001. 

\bibitem{mcc98} Madan, D. B., Carr, P. and Chang, E. C.  The variance gamma process and option pricing. \emph{Eur. Finance Rev.} $\mathbf{2}$ (1998),  74--105.

\bibitem{madan} Madan, D. B. and Seneta, E. The Variance Gamma (V.G.) Model for Share Market Returns.  \emph{J. Bus.} $\mathbf{63}$ (1990),  511--524.


\bibitem{np16} Nadarajah, S. and Pog\'{a}ny, T. K. On the distribution of the product of correlated normal random variables. \emph{C.R. Acad. Sci. Paris, Ser. I} $\mathbf{354}$ (2016),  201--204.

\bibitem{o21} Ogasawara, H. A non-recursive formula for various moments of the multivariate normal distribution with sectional truncation. \emph{J. Multivariate Anal.} $\mathbf{183}$ (2021), Art.\ 104792.

\bibitem{olver} Olver, F. W. J., Lozier, D. W., Boisvert, R. F. and Clark, C. W.  \emph{NIST Handbook of Mathematical Functions.} Cambridge University Press, 2010.


\bibitem{scott} Scott, D. J., W\"urtz, D., Dong, C. and Tran, T. T.  Moments of the generalized hyperbolic distribution. \emph{Computation. Stat.} $\mathbf{26}$ (2011),  459--476.

\bibitem{s04} Seneta, E.  Fitting the Variance-Gamma Model to Financial Data. \emph{J. Appl. Probab.} $\mathbf{41}$ (2004),  177--187.

\bibitem{s73} Sichel, H. S.  Statistical valuation of diamondiferous deposits. \emph{J. S. Afr. Inst. Min. Metall.} $\mathbf{73}$ (1973),  235--243.


\bibitem{w12} Winkelbauer, A. Moments and Absolute Moments of the Normal Distribution. arXiv:1209.4340, 2012.

\bibitem{wb32} Wishart, J. and Bartlett, M. S. The distribution of second order moment statistics in a normal system. \emph{Proc. Camb. Philol. Soc.} $\mathbf{28}$ (1932), 455--459.

\end{thebibliography}
\end{document}